\documentclass{article}
\usepackage{amssymb}
\usepackage{amsmath}
\usepackage{amsfonts}

\newtheorem{theorem}{Theorem}

\newtheorem{corollary}[theorem]{Corollary}

\newtheorem{definition}[theorem]{Definition}
\newtheorem{example}[theorem]{Example}

\newtheorem{lemma}[theorem]{Lemma}

\newtheorem{proposition}[theorem]{Proposition}
\newtheorem{remark}[theorem]{Remark}

\newenvironment{proof}[1][Proof]{\noindent\textbf{#1.} }{\ \rule{0.5em}{0.5em}}
\input{tcilatex}
\begin{document}

\begin{center}
\renewcommand{\thefootnote}{}{\large On a
remarkable class of left-symmetric algebras and its relationship with the
class of Novikov algebras}

\bigskip\ 

\bigskip {\large M. Guediri}

\vspace{0.5in}
\end{center}

$\mathtt{Abstract.}$\emph{\ }\textsl{We discuss locally simply transitive
affine actions of Lie groups }$G$ \textsl{on finite-dimensional vector spaces%
} \textsl{such that the commutator subgroup\ }$\left[ G,G\right] $ \textsl{%
is acting by translations. In other words, we consider left-symmetric
algebras satisfying the identity }$\left[ x,y\right] \cdot z=0.$\textsl{\ We
derive some basic characterizations of such left-symmetric algebras and we
highlight their relationships with the so-called Novikov algebras and
derivation algebras.}

\section{Introduction}

Let $M$ be a differentiable manifold together with a flat torsion-free
connection $\nabla ,$ and let $\mathfrak{X}\left( M\right) $ denote the
space of all vector fields on $M.$ Endowing $\mathfrak{X}\left( M\right) $
with the product defined by $X\cdot Y=\nabla _{X}Y,$ we observe that the
flatness of $\nabla $ yields the identity%
\begin{equation}
\left( X\cdot Y\right) \cdot Z-X\cdot \left( Y\cdot Z\right) =\left( Y\cdot
X\right) \cdot Z-Y\cdot \left( X\cdot Z\right) ,  \label{eq1}
\end{equation}%
for all $X,Y,Z\in \mathfrak{X}\left( M\right) .$

A finite-dimensional vector space together with a product satisfying the
above identity is called \emph{a left-symmetric algebra.} The concept of a
left-symmetric algebra has its roots in the paper \cite{cartan}, and the
term was first introduced in \cite{vinberg1} (see also \cite{vinberg2}).

\smallskip

If $\left( A,\cdot \right) $ is a left-symmetric algebra over a field $%
\mathbb{F}$, then the binary operation defined by $\left[ x,y\right] =x\cdot
y-y\cdot x$ makes $A$ into a Lie algebra called the associated Lie algebra.
Conversely, if $\mathcal{G}$ is a Lie algebra together with a left-symmetric
product such that $\left[ x,y\right] =x\cdot y-y\cdot x,$ then we say that
this left-symmetric product is compatible with the Lie structure of $%
\mathcal{G}.$ For instance, if $\nabla $ is a flat torsion-free connection
on a differentiable manifold $M,$ then it is known that the space $\mathfrak{%
X}\left( M\right) $ is an infinite dimensional Lie algebra under the
standard Lie bracket of vector fields. In this case, the left-symmetric
product on $\mathfrak{X}\left( M\right) $ defined by $X\cdot Y=\nabla _{X}Y$
is compatible with the standard Lie structure of $\mathfrak{X}\left(
M\right) .$

\smallskip

Let now $M=G$ be a finite dimensional Lie group with Lie algebra $\mathcal{G}%
,$ and let $\nabla $ be a left-invariant affine connection (i.e., a
left-invariant flat torsion-free connection) on $G.$ As mentioned above, by
identifying the Lie algebra $\mathcal{G}$ of $G$ with the set of all
left-invariant vector fields on $G,$ we see that $\mathcal{G}$ becomes a
finite dimensional left-symmetric subalgebra of $\mathfrak{X}\left( G\right)
.$ Conversely, if $\mathcal{G}$ is endowed with a left-symmetric product
which is compatible with the Lie structure of $\mathcal{G},$ then by means
of the conditions 
\begin{equation*}
\left( fX\right) \cdot Y=f\left( X\cdot Y\right) ,\ \ \ \ \ \ \ X\cdot
\left( fY\right) =X\left( f\right) \cdot Y+f\left( X\cdot Y\right) ,
\end{equation*}%
for all $X,Y\in \mathfrak{X}\left( G\right) $ and $f\in C^{\infty }\left(
G\right) ,$ we can extend this left-symmetric product to a left-invariant
affine connection on $G.$

To summarize, we deduce that giving a left-invariant affine connection on a
Lie group $G$ is precisely equivalent to giving a compatible left-symmetric
product on its Lie algebra $\mathcal{G}.$ Moreover it turns out that, under
this equivalence, bi-invariant affine connections on $G$ correspond exactly
to associative produtcs on $\mathcal{G}$ (cf. \cite{milnor}, p. 186; see
also \cite{medina}, Proposition 1.1).

\smallskip

Another geometric source of left-symmetric algebras comes from the geometry
of locally simply transitive actions of Lie groups by affine transformations
on vector spaces. It turns out that a simply connected Lie group $G$ admits
a left-invariant affine connection if and only if $G$ acts locally simply
transitively on a vector space $E$ by affine transformations (cf. \cite{kim}%
). Accordingly, if $G$ is a simply connected Lie group with Lie algebra $%
\mathcal{G},$ then giving a locally simply transitive affine action of $G$
on a vector space $E$ can be interpreted as giving a compatible
left-symmetric product on $\mathcal{G}.$ In this respect, it has been proved
in \cite{auslander} that a simply connected Lie group $G$ which acts simply
transitively by affine transformations of $\mathbb{R}^{n}$ is necessarily
solvable. This result leads to the question raised in \cite{milnor} of
whether every solvable simply connected Lie group $G$ can act simply
transitively by affine transformations of $\mathbb{R}^{n}.$ A negative
answer to this question can be found in \cite{benoist}.

\bigskip

The object of this paper is to introduce a new class of left-symmetric
algebras, namely, the class of finite-dimensional left-symmetric algebras $A$
over a field $\mathbb{F}$ satisfying the identity%
\begin{equation}
\left( x\cdot y\right) \cdot z=\left( y\cdot x\right) \cdot z,\ \ \ \text{%
for all }x,y,z\in A,  \label{eq4}
\end{equation}%
or, equivalently, 
\begin{equation}
\left[ x,y\right] \cdot z=0,\ \ \ \text{for all }x,y,z\in A.  \label{eq4a}
\end{equation}

In terms of affine actions, this is equivalent to consider locally simply
transitive affine actions of Lie groups $G$ on a vector space $E$ such that
the commutator subgroup $\left[ G,G\right] $\ is acting by translations.
Left-symmetric algebras over $\mathbb{R}$ satisfying the identity (\ref{eq4}%
) are classified up to dimesnsion $4$ in \cite{kholoud}.

\smallskip

The paper is organized as follows. In Section 2, after recalling the
definition of left-symmetric algebras, we give the definitions of the
so-called Novikov algebras and derivation algebras. These are defined as
left-symmetric algebras satisfying the identities $\left( x\cdot y\right)
\cdot z=\left( x\cdot z\right) \cdot y$ and $\left( x\cdot y\right) \cdot
z=\left( z\cdot y\right) \cdot x,$ respectively. We then consider
left-symmetric algebras satisfying the identity (\ref{eq4}). We give some
examples showing that these three conditions are pairwise independent, and
we prove that if any two of these three conditions are satisfied for a given
left-symmetric algebra then the third condition is satisfied as well. We
also give a criterion that connects associative algebras satisfying (\ref%
{eq4}) and Novikov assocative algebras, and we deduce that the Lie algebra
associated to a Novikov assocative algebra is necessarily $2$-step solvable.

In Section 3, we introduce the notion of center for a left-symmetric algebra
and we turn our attention to the question as to when this center coincides
with the center of the associated Lie algebra. For instance, we show that
these two centers coincide if the left-symmetric algebra is Novikov or
derivation, and we give an example of a left-symmetric algebra satisfying (%
\ref{eq4}) so that its center does not coincide with the center of its
associated Lie algebra. We also show that the center of a left-symmetric
algebra that is Novikov or derivation is a two-sided ideal, and we give an
example of a left-symmetric algebra satisfying the (\ref{eq4}) (and which is
of course neither Novikov nor derivation) so that its center fails to be a
two-sided ideal. In the last two subsections of this section, we introduce
the translational center of a left-symmetric algebra. In terms of affine
actions, this corresponds to the set of translations lying in the center of
the group. We show that every left-symmetric algebra satisfying (\ref{eq4})
whose associated Lie algebra is nilpotent can be obtained as a central
extension of a smaller dimensional left-symmetric algebra satisfying (\ref%
{eq4}). We also show that the translational center of a noncommutative
derivation algebra is non-trivial, and we deduce from this that a derivation
algebra whose associated Lie algebra is nonsingular nilpotent is necessarily
complete, that is, all right multiplications are nilpotent transformations.

In Section 4, we introduce two notions of a radical for a left-symmetric
algebra $A$, namely the Koszul radical $R\left( A\right) $ and the
right-nilpotent radical $N\left( A\right) .$ If $A$ is a left-symmetric
algebra satisfying (\ref{eq4}), we show that $R\left( A\right) $ is a
two-sided ideal of $A$ containing the derived Lie algebra $\left[ A,A\right]
.$ As a consequence, we conclude that every left-symmetric algebra
satisfying (\ref{eq4}) can be obtained as an extension of a commutative
associative algebra by a complete left-symmetric algebra satisfying (\ref%
{eq4}). This allows us to classify all low-dimensional left-symmetric
algebras satisfying (\ref{eq4}). If $A$ is a Novikov algebra with associated
Lie algebra $\mathcal{G}_{A},$ we show that these two radicals coincide and
are equal to the kernel of the homomorphism $L:A\rightarrow \mathfrak{gl}%
\left( \mathcal{G}_{A}\right) ,$ where $L_{x}$ is the left-multiplication
defined by $L_{x}y=x\cdot y.$

Finally, in Section 5, we make some comments on simple left-symmetric
algebras. In \cite{zelmanov}, Zelmanov proved that a finite-dimensional
simple Novikov algebra $A$ over a field $\mathbb{F}$ of characteristic zero
is one-dimensional (i.e., $A$ coincides with $\mathbb{F}$). This is true
when $\mathbb{F}$ is algebraically closed. However, the two-dimensional
commutative algebra $A_{2,\mathbb{R}}$ over $\mathbb{R}$ defined by the
following multiplication table:$\ e_{1}\cdot e_{1}=e_{1},\ \ e_{1}\cdot
e_{2}=e_{2}\cdot e_{1}=e_{2},\ \ e_{2}\cdot e_{2}=-e_{1}$ is a Novikov
algebra which turns out to be simple. In fact, we shall show that this is
the unique simple Novikov algebra over $\mathbb{R}$ of dimension $\geq 2.$
We shall also show that a simple left-symmetric algebra over $\mathbb{R}$
which is derivation or satisfying(\ref{eq4}) is isomorphic to either $A_{2,%
\mathbb{R}}$ or the field $\mathbb{R}.$ We close this section by pointing
out that every complete left-symmetric algebra over a field of
characteristic zero that is Novikov, derivation, or satisfying (\ref{eq4})
is not simple.

\section{Left-symmetric algebras}

Throughout this paper, $\mathbb{F}$ will denote a field of characteristic
zero. As we mentioned above, recall that a finite-dimensional algebra $%
\left( A,\cdot \right) $ over $\mathbb{F}$ is called \emph{left-symmetric}
if it satisfies the identity%
\begin{equation*}
\left( x,y,z\right) =\left( y,x,z\right) ,\ \ \ \text{for all }x,y,z\in A,
\end{equation*}%
where $\left( x,y,z\right) $ denotes the associator $\left( x,y,z\right)
=\left( x\cdot y\right) \cdot z-x\cdot \left( y\cdot z\right) .$ In this
case, the commutator $\left[ x,y\right] =x\cdot y-y\cdot x$ defines a
bracket that makes $A$ into a Lie algebra. We denote by $\mathcal{G}_{A}$
this Lie algebra and call it \emph{the associated Lie algebra to }$A.$
Conversely, if $\mathcal{G}$ is a Lie algebra endowed with a left-symmetric
product satisfying the condition $\left[ x,y\right] =x\cdot y-y\cdot x,$
then we say that this left-symmetric product is \emph{compatible} with the
Lie structure of $\mathcal{G}.$

Let $A$ be a left-symmetric algebra over a field $\mathbb{F},$ and let $%
L_{x} $ and $R_{x}$ denote the left and right multiplications by the element 
$x\in A$, respectively. The identity (\ref{eq1}) is now equivalent to the
formula 
\begin{equation}
\left[ L_{x},L_{y}\right] =L_{\left[ x,y\right] },\ \ \ \text{for all }%
x,y\in A,  \label{eq1a}
\end{equation}%
or, in other words, the linear map $L:\mathcal{G}_{A}\rightarrow End\left(
A\right) $ is a representation of Lie algebras.

We notice that (\ref{eq1}) is also equivalent to the formula 
\begin{equation}
\left[ L_{x},R_{y}\right] =R_{x\cdot y}-R_{y}\circ R_{x},\ \ \ \text{for all 
}x,y\in A.  \label{eq1b}
\end{equation}

We say that $A$ is \emph{complete} if $R_{x}$ is a nilpotent operator, for
all $x\in A.$ In this context, if $G$ is an $n$-dimensional simply connected
Lie group with Lie algebra $\mathcal{G},$ then giving a compatible complete
left-symmetric product on $\mathcal{G}$ can be interpreted as giving a
complete left-invariant affine connection on $G$ or, equivalently, as giving
a simply transitive affine action of $G$ on an $n$-dimensional vector space $%
E.$

We say that $A$ is \emph{a Novikov algebra} if it satisfies the identity%
\begin{equation}
\left( x\cdot y\right) \cdot z=\left( x\cdot z\right) \cdot y,\ \ \ \text{%
for all }x,y,z\in A.  \label{eq2}
\end{equation}

In terms of left and right multiplications, (\ref{eq2}) is equivalent to any
one of the following formulas%
\begin{eqnarray}
\left[ R_{x},R_{y}\right] &=&0,\ \ \ \text{for all }x,y\in A,  \label{eq2a}
\\
\ L_{x\cdot y} &=&R_{y}\circ L_{x},\ \ \ \text{for all }x,y\in A.
\label{eq2b}
\end{eqnarray}

The left-symmetric algebra $A$ is called \emph{a derivation algebra} if it
satisfies the identity 
\begin{equation}
\left( x\cdot y\right) \cdot z=\left( z\cdot y\right) \cdot x,\ \ \ \text{%
for all }x,y,z\in A.  \label{eq3}
\end{equation}%
or, equivalently, all left and right multiplications $L_{x}$ and $R_{x}$ are
derivations of the ssociated Lie algebra $\mathcal{G}_{A}.$ This is also
equivalent to the formula%
\begin{equation}
L_{x\cdot y}=R_{x}\circ R_{y},\ \ \ \text{for all }x,y\in A.  \label{eq3a}
\end{equation}

\subsection{Left-symmetric algebras satisfying $\left[ x,y\right] \cdot z=0$}

As we mentioned before, in this paper we are mainly concerned with locally
simply transitive affine actions of Lie groups\ $G$ on vector spaces such
that the commutator subgroup $\left[ G,G\right] $\ is acting by
translations. In terms of left-symmetric structures, we are concerned with
the class of finite-dimensional left-symmetric algebras satisfying the
identity (\ref{eq4}). If $A$ is such an algebra, then recall that (\ref{eq4}%
) is equivalent to the identity $\left[ x,y\right] \cdot z=0$ for all $%
x,y,z\in A$, which in turn is equivalent to say that the restriction of any
right multiplication $R_{x}$ to the derived ideal $\left[ A,A\right] $ is
identically zero. In light of (\ref{eq1a}), this is also equivalent to the
identity

\begin{equation}
\left[ L_{x},L_{y}\right] =0,\ \ \ \text{for all }x,y\in A.  \label{eq4b}
\end{equation}

The simplest examples of left-symmetric algebras satisfying (\ref{eq4}) are
the commutative algebras. This is stated in the following lemma, whose proof
is straightforward.

\begin{lemma}
A left-symmetric algebra $A$ over a field $\mathbb{F}$ is commutative if and
only if the associated Lie algebra $\mathcal{G}_{A}$ is abelian. In that
case, $A$ is associative, Novikov, derivation, and satisfying (\ref{eq4}).
\end{lemma}

However, an arbitrary left-symmetric algebra which satisfies (\ref{eq4})
need not be Novikov nor derivation at all. In fact, the conditions (\ref{eq2}%
), (\ref{eq3}), and (\ref{eq4}) are pairwise independent, as the following
examples show.

\begin{example}
Over any field $\mathbb{F},$ there is up to isomorphism just one
two-dimensional non-abelian Lie algebra, namely the Lie algebra $\mathcal{G}%
_{2}$ of the affine group of the line. It has a basis $\left\{
e_{1},e_{2}\right\} $ such that $\left[ e_{1},e_{2}\right] =e_{2}.$ Let $%
A_{21}$ be the left-symmetric algebra whose associated Lie algebra is $%
\mathcal{G}_{2}$ and which has multiplication table: $e_{1}\cdot
e_{1}=e_{1}, $ $e_{1}\cdot e_{2}=e_{2}.$ It is easy to check that $A_{21}$
satisfies (\ref{eq4}) but $A_{21}$ is neither Novikov nor derivation. Of
course, $A_{21}$ is not complete. We shall give later an example of a
complete left-symmetric algebra satisfying (\ref{eq4}) which is neither
Novikov nor derivation (see Example \ref{exp4} below).
\end{example}

\begin{example}
Let $\mathcal{G}_{2}$ be as in the example above, and let $A_{22}$ be the
left-symmetric algebra whose associated Lie algebra is $\mathcal{G}_{2}$ and
which has multiplication table: $e_{1}\cdot e_{1}=-e_{1},$ $e_{2}\cdot
e_{1}=-e_{2}.$ It is easy to check that $A_{22}$ is a Novikov algebra which
is not derivation and does not satisfy (\ref{eq4}).
\end{example}

\begin{example}
Let $\mathcal{G}_{3}$ be the real Lie algebra with basis $\left\{
e_{1},e_{2},e_{3}\right\} $ and non-trivial Lie brackets $\left[ e_{1},e_{2}%
\right] =e_{2},$ $\left[ e_{1},e_{3}\right] =e_{3}.$ Let $A_{3,\gamma }$ be
the left-symmetric algebra whose associated Lie algebra is $\mathcal{G}_{3}$
and which has multiplication table: $e_{1}\cdot e_{1}=\alpha e_{2}+\beta
e_{3},$ $e_{1}\cdot e_{2}=e_{2}+\gamma e_{3},$ $e_{2}\cdot e_{1}=\gamma
e_{3},$ where $\alpha ,\beta ,\gamma $ are arbitrary real numbers. It is
easy to check that $A_{3,\gamma }$ is a derivation algebra for any $\gamma $
and that $A_{3,\gamma }$ satisfies (\ref{eq4}) if and only if $\gamma =0.$
It therefore follows from Proposition \ref{prop1} below that $A_{3,\gamma }$
is Novikov if and only if $\gamma =0.$ Hence, assuming $\gamma \neq 0,$ we
obtain a derivation left-symmetric algebra $A_{3,\gamma }$ which is not
Novikov and does not satisfy (\ref{eq4}).
\end{example}

Although identities (\ref{eq2}), (\ref{eq3}), and (\ref{eq4}) are pairwise
independent, we have the following remarkable fact.

\begin{proposition}
\label{prop1}Let $A$ be a left-symmetric algebra over a field $\mathbb{F}.$
If any two of the conditions (\ref{eq2}), (\ref{eq3}), (\ref{eq4}) are
satisfied, then the third condition is satisfied as well.
\end{proposition}

\begin{proof}
Assume first that $A$ satisfies (\ref{eq2}) and (\ref{eq3}), that is $A$ is
Novikov and derivation. Using (\ref{eq2}), (\ref{eq3}), and once again (\ref%
{eq2}), it follows that 
\begin{equation*}
\left( x\cdot y\right) \cdot z=\left( x\cdot z\right) \cdot y=\left( y\cdot
z\right) \cdot x=\left( y\cdot x\right) \cdot z,
\end{equation*}%
that is (\ref{eq4}) is satisfied. The other cases can be proved in a similar
way.
\end{proof}

\bigskip

In \cite{giraud-medina}, it was shown that the Lie algebra associated to a
derivation algebra is necessarily solvable. In the case of Novikov algebras,
a similar result can be deduced from \cite{zelmanov}. In the case of
left-symmetric algebras satisfying (\ref{eq4}), we can be more specific.

\begin{proposition}
\label{prop2}Let $A$ be a left-symmetric algebra satisfying (\ref{eq4}).
Then, the associated Lie algebra $\mathcal{G}_{A}$ is 2-step solvable.
\end{proposition}

\begin{proof}
By (\ref{eq4a}), $L_{x}=0$ for all $x\in \mathcal{D}\mathcal{G}_{A}=\left[ 
\mathcal{G}_{A},\mathcal{G}_{A}\right] .$ It follows that $\left[ x,y\right]
=x\cdot y-y\cdot x=L_{x}y-L_{y}x=0$ for all $x,y\in \mathcal{D}\mathcal{G}%
_{A}.$ Therefore, $\mathcal{D}\mathcal{G}_{A}$ is abelian, that is $\mathcal{%
G}_{A}$ is 2-step solvable.
\end{proof}

\smallskip

The following crucial observation is an immediate consequence of (\ref{eq4a}%
).

\begin{proposition}
\label{prop2bis}Let $A$ be a left-symmetric algebra over a field $\mathbb{F}$
satisfying (\ref{eq4}). Then $\left[ A,A\right] $ is a two-sided ideal of $%
A. $
\end{proposition}

\subsection{Associative algebras satisfying $\left[ x,y\right] \cdot z=0$}

Let $\left( A,\cdot \right) $ be an arbitrary algebra over a field $\mathbb{F%
}.$ Define a new product $\circ $ on $A$ as follows: 
\begin{equation*}
x\circ y=-y\cdot x,\ \ \ \text{for all }x,y\in A.
\end{equation*}

It is obvious that $\left( A,\circ \right) $ is an algebra over $\mathbb{F}$%
, and that $\left( A,\cdot \right) $ is left (resp. right)-symmetric if and
only if $\left( A,\circ \right) $ right (resp. left)-symmetric, where right
symmetry means here that $\left( x,y,z\right) =\left( x,z,y\right) $ for all 
$x,y,z\in A.$ Particularly, we have that $\left( A,\cdot \right) $ is
associative if and only if $\left( A,\circ \right) $ is so. In that case, we
have the following:

\begin{proposition}
\label{prop3}Let $\left( A,\cdot \right) $ be an associative algebra over a
field $\mathbb{F},$ and let $\left( A,\circ \right) $ be the associative
algebra over $\mathbb{F}$ defined as above. Then, $\left( A,\cdot \right) $
satisfies (\ref{eq4}) if and only if $\left( A,\circ \right) $ is Novikov.
\end{proposition}

\begin{proof}
Let $\left( A,\cdot \right) $ be an associative algebra, and define the
associative product $\circ $ on $A$ as above. Using the associativity of the
product $\circ ,$ we have 
\begin{eqnarray*}
\left[ x,y\right] \cdot z &=&\left( x\cdot y\right) \cdot z-\left( y\cdot
x\right) \cdot z \\
&=&z\circ \left( y\circ x\right) -z\circ \left( x\circ y\right) \\
&=&\left( z\circ y\right) \circ x-\left( z\circ x\right) \circ y,\ \ \ \text{%
for all }x,y,z\in A.
\end{eqnarray*}

It follows from this equality that $\left( A,\cdot \right) $ satisfies (\ref%
{eq4}) if and only if $\left( z\circ y\right) \circ x-\left( z\circ x\right)
\circ y=0\ $for all $x,y,z\in A$, that is $\left( A,\circ \right) $ is
Novikov, as desired.
\end{proof}

\bigskip

As a consequence of Proposition \ref{prop2} and Proposition \ref{prop3}, we
have the following:

\begin{corollary}
Let $\mathcal{G}$ be a Lie algebra which can admit an associative Novikov
structure. Then, $\mathcal{G}$ is $2$-step solvable.
\end{corollary}

\section{\protect\bigskip The center of a left-symmetric algebra}

Recall that the center $Z\left( A\right) $ of an arbitrary algebra $A$ (see
for instance \cite{schafer}) is defined to be the set of all $z$ in $A$
which commute and associate with all elements in $A;$ that is, the set of
all $z$ in $A$ with the property $x\cdot z=z\cdot x$ for all $x\in A,$ and
which satisfy the additional identity 
\begin{equation}
\left( z,x,y\right) =\left( x,z,y\right) =\left( x,y,z\right) =0,\ \text{for
all }x,y\in A.  \label{eq5}
\end{equation}

We first begin by showing that if $A$ is a left-symmetric algebra and $%
\mathcal{Z}_{A}$ is the center of the Lie algebra $\mathcal{G}_{A}$
associated to $A,$ then (\ref{eq5}) simplifies.

\begin{lemma}
\label{lemma1}Let $A$ be a left-symmetric algebra. Then 
\begin{equation*}
Z\left( A\right) =\left\{ z\in \mathcal{Z}_{A}:\left( z,x,y\right) =0,~\ \ 
\text{for all }x,y\in A.\right\} .
\end{equation*}
\end{lemma}

\begin{proof}
By definition, $Z\left( A\right) \subseteq \mathcal{Z}_{A}.$ Conversely, for
all $z\in \mathcal{Z}_{A}$, we have 
\begin{eqnarray*}
\left( x,y,z\right) &=&\left( x\cdot y\right) \cdot z-x\cdot \left( y\cdot
z\right) \\
&=&z\cdot \left( x\cdot y\right) -x\cdot \left( z\cdot y\right) \\
&=&\left( z\cdot x\right) \cdot y-\left( x\cdot z\right) \cdot y \\
&=&\left( x\cdot z\right) \cdot y-\left( x\cdot z\right) \cdot y \\
&=&0.
\end{eqnarray*}%
The lemma follows now by observing that $\left( z,x,y\right) =\left(
x,z,y\right) $ for all $x,y,z\in A.$
\end{proof}

\bigskip

In case $A$ is a derivation algebra, it turns out that $Z\left( A\right) $
coincides with $\mathcal{Z}_{A}$ (see \cite{medina}, 2.2.7). We show here
that this is also true for a Novikov algebra.

\begin{lemma}
\label{lemma2}If $A$ is a Novikov or derivation algebra, then $Z\left(
A\right) =\mathcal{Z}_{A}.$
\end{lemma}

\begin{proof}
Assume that $A$ is a Novikov algebra. Then, according to Lemma \ref{lemma1},
we only have to show that $\left( z,x,y\right) =0$ for all $z\in \mathcal{Z}%
_{A}$ and $x,y\in A.$ To do this, we use identity (\ref{eq2}) to obtain 
\begin{eqnarray*}
\left( z,x,y\right) &=&\left( z\cdot x\right) \cdot y-z\cdot \left( x\cdot
y\right) \\
&=&\left( z\cdot x\right) \cdot y-\left( x\cdot y\right) \cdot z \\
&=&\left( z\cdot x\right) \cdot y-\left( x\cdot z\right) \cdot y \\
&=&\left( z\cdot x\right) \cdot y-\left( z\cdot x\right) \cdot y \\
&=&0.
\end{eqnarray*}
\end{proof}

\bigskip

If $A$ is a left-symmetric algebra satisfying (\ref{eq4}), then $Z\left(
A\right) $ can be different from $\mathcal{Z}_{A}.$ Here is an example.

\begin{example}
\label{exp4}On the Lie algebra $\mathcal{H}_{3}\times \mathbb{R}$ with basis 
$\left\{ e_{1},e_{2},e_{3},e_{4}\right\} $ such that $\left[ e_{2},e_{3}%
\right] =e_{1},$ define a left-symmetric product as follows: $e_{2}\cdot
e_{3}=e_{3}\cdot e_{4}=e_{4}\cdot e_{3}=e_{1},$\ $e_{4}\cdot e_{4}=e_{2}.$
The resulting left-symmetric algebra that we denote by $A$ satisfies (\ref%
{eq4}), since $\left[ A,A\right] =\mathbb{R}e_{1}$ and $L_{e_{1}}=0.$
However, it is easy to verify that $A$ is neither Novikov nor derivation.
Now, we notice that $\mathcal{Z}_{A}=span\left\{ e_{1},e_{4}\right\} ;$ and
to show that $Z\left( A\right) \neq \mathcal{Z}_{A}$ we need to check that $%
e_{4}\notin Z\left( A\right) .$ To this end, it suffices to show that $%
\left( e_{4},x,y\right) \neq 0$ for some $x,y\in A.$ Indeed, as desired, we
have 
\begin{eqnarray*}
\left( e_{4},e_{4},e_{3}\right) &=&\left( e_{4}\cdot e_{4}\right) \cdot
e_{3}-e_{4}\cdot \left( e_{4}\cdot e_{3}\right) \\
&=&e_{2}\cdot e_{3}-e_{4}\cdot e_{1} \\
&=&e_{1}.
\end{eqnarray*}
\end{example}

\begin{remark}
\label{remark2}It is remarkable that (\ref{eq4a}) (or equivalently (\ref{eq4}%
) ) is partially true for Novikov algebras and derivation algebras in the
sense that, in both these cases, $R_{z}$ is identically zero on the derived
ideal $\left[ A,A\right] ,$ for all $z$ in the center $Z\left( A\right) $ of 
$A.$
\end{remark}

It is clear that $Z\left( A\right) $ is always a commutative associative
subalgebra of $A.$ However, even if $A$ is a left-symmetric algebra, $%
Z\left( A\right) $ need not be a left (resp. right or two-sided) ideal of $%
A. $ An example is the following:

\begin{example}
Let $A_{4}$ be the four-dimensional real left-symmetric algebra given by the
following multiplication table:%
\begin{equation*}
e_{1}\cdot e_{4}=e_{4}\cdot e_{1}=e_{2},\ \ e_{2}\cdot e_{3}=e_{1},\ \
e_{2}\cdot e_{4}=-e_{3},\ \ e_{3}\cdot e_{3}=e_{2}.
\end{equation*}

It is clear that $\mathcal{Z}_{A_{4}}=\mathbb{R}e_{1}$ is neither a left
ideal nor a right ideal of $A_{4},$ and that $Z\left( A_{4}\right) =\mathcal{%
Z}_{A_{4}}.$
\end{example}

\begin{remark}
It is easy to check that the left-symmetric algebra $A_{4}$ of the above
example does not satisfy (\ref{eq4}) and is neither Novikov nor derivation.
\end{remark}

Now, we specialize to the case when $A$ is a Novikov algebra or a derivation
algebra. In that case, we have the following result.

\begin{proposition}
If $A$ is a Novikov algebra or a derivation algebra, then $Z\left( A\right) =%
\mathcal{Z}_{A}$ is a two-sided ideal of $A.$
\end{proposition}

\begin{proof}
For an arbitrary left-symmetric algebra $A,$ identity (\ref{eq1}) is clearly
equivalent to the identity 
\begin{equation*}
\left[ x,y\right] \cdot z=x\cdot \left( y\cdot z\right) -y\cdot \left(
x\cdot z\right) ,\ \ \ x,y,z\in A.
\end{equation*}

Now, by assuming that $A$ is a Novikov algebra or a derivation algebra and
using Lemma \ref{lemma2} and Lemma \ref{remark2}, we deduce that $Z\left(
A\right) =\mathcal{Z}_{A}$ and $\left[ x,y\right] \cdot z=0$ for all $x,y\in
A$ and $z\in Z\left( A\right) .$ Thus, the above identity reduces to the
identity%
\begin{equation}
x\cdot \left( y\cdot z\right) =y\cdot \left( x\cdot z\right) ,\ \ x,y\in A%
\text{\ and }z\in Z\left( A\right) .\   \label{eq6}
\end{equation}

In case $A$ is Novikov, using (\ref{eq6}) we get%
\begin{eqnarray*}
\left( x\cdot z\right) \cdot y-x\cdot \left( z\cdot y\right) &=&\left(
z\cdot x\right) \cdot y-z\cdot \left( x\cdot y\right) \\
&=&\left( z\cdot x\right) \cdot y-\left( x\cdot y\right) \cdot z \\
&=&\left( z\cdot x\right) \cdot y-\left( x\cdot z\right) \cdot y \\
&=&0,
\end{eqnarray*}%
for all $x,y\in A\ $and $z\in Z\left( A\right) .$

In case $A$ is derivation, using (\ref{eq6}) again we get%
\begin{eqnarray*}
\left( x\cdot z\right) \cdot y-x\cdot \left( z\cdot y\right) &=&\left(
z\cdot x\right) \cdot y-z\cdot \left( x\cdot y\right) \\
&=&\left( z\cdot x\right) \cdot y-\left( x\cdot y\right) \cdot z \\
&=&\left( y\cdot x\right) \cdot z-\left( x\cdot y\right) \cdot z \\
&=&\left[ y,x\right] \cdot z \\
&=&0,
\end{eqnarray*}%
for all $x,y\in A\ $and $z\in Z\left( A\right) .$

In summary, we have shown that if $A$ is a Novikov algebra or a derivation
algebra, then 
\begin{equation}
\left( x\cdot z\right) \cdot y=x\cdot \left( z\cdot y\right) ,\ \ x,y\in A%
\text{\ and }z\in Z\left( A\right) .  \label{eq7}
\end{equation}

Assume now that $A$ is a Novikov algebra or a derivation algebra, and let $%
x,y\in A$\ and $z\in Z\left( A\right) .$ Then, by combining (\ref{eq6}) with
(\ref{eq7}), we obtain 
\begin{eqnarray*}
\left( x\cdot z\right) \cdot y &=&x\cdot \left( z\cdot y\right) \\
&=&x\cdot \left( y\cdot z\right) \\
&=&y\cdot \left( x\cdot z\right) ,
\end{eqnarray*}%
which is equivalent to the identity $\left[ x\cdot z,y\right] =0,$ from
which we conclude that $x\cdot z\in \mathcal{Z}_{A}=Z\left( A\right) .$ That
is, $Z\left( A\right) =\mathcal{Z}_{A}$ is a left ideal of $A$.

Since $z\cdot x=x\cdot z,$ we conclude that $Z\left( A\right) =\mathcal{Z}%
_{A}$ is also a right ideal of $A$. Hence $Z\left( A\right) =\mathcal{Z}_{A}$
is a two-sided ideal of $A,$ as desired.
\end{proof}

\bigskip

If $A$ is a left-symmetric algebra satisfying (\ref{eq4}) which is neither
Novikov nor derivation, then $Z\left( A\right) $ need not be a two-sided
ideal of $A.$ Here is an example.

\begin{example}
On the Lie algebra $\mathcal{H}_{3}\times \mathbb{R}$ with basis $\left\{
e_{1},e_{2},e_{3},e_{4}\right\} $ such that $\left[ e_{2},e_{3}\right]
=e_{1},$ define a left-symmetric product as follows: $e_{2}\cdot
e_{3}=e_{1},\ e_{3}\cdot e_{2}=2e_{1},\ e_{3}\cdot e_{3}=e_{4},$\ $%
e_{3}\cdot e_{4}=e_{4}\cdot e_{3}=e_{2},\ e_{4}\cdot e_{4}=2e_{1}.$ The
resulting left-symmetric algebra that we denote by $A$ satisfies (\ref{eq4}%
), since $\left[ A,A\right] =\mathbb{R}e_{1}$ and $L_{e_{1}}=0.$ However, it
is easy to check that $A$ is neither Novikov nor derivation. It is also easy
to verify that $Z\left( A\right) =\mathcal{Z}_{A}=span\left\{
e_{1},e_{4}\right\} ,$ and that $Z\left( A\right) $ is neither a left ideal
nor a right ideal of $A$
\end{example}

\subsection{The translational center of a left-symmetric algebra}

Given a left-symmetric algebra $A$ over a field $\mathbb{F},$ we consider
the set 
\begin{equation*}
T\left( A\right) =\left\{ x\in A:x\cdot y=0,\ \ \ \text{for all }y\in
A\right\} .
\end{equation*}

Since $T\left( A\right) $ is both a right ideal of $A$ and a Lie ideal of
the associated Lie algebra $\mathcal{G}_{A},$ it follows that $T\left(
A\right) $ is a two-sided ideal of $A.$

\begin{definition}
The subset $C\left( A\right) $ of $A$ given by%
\begin{equation*}
C\left( A\right) =T\left( A\right) \cap Z\left( A\right)
\end{equation*}%
will be called the \emph{translational center} of $A.$
\end{definition}

We note that, in \cite{kim}, the center of $A$ is defined to be the subset $%
T\left( A\right) \cap \mathcal{Z}_{A},$ but it follows from Lemma \ref%
{lemma1} that this coincides with $C\left( A\right) .$ On the other hand,
the translational center has the following geometric interpretation:

We know that a transitive (resp. simply transitive) affine action of a Lie
group $G$ on $\mathbb{R}^{n}$ corresponds to the existence of a
left-symmetric (resp. complete left-symmetric) product on the Lie algebra $%
\mathcal{G}$ of $G$ (see \cite{kim}), and it turns out that the exponential
of $C\left( A\right) $ corresponds exactly to central translations in $G.$
For instance, we know by \cite{fried-goldman} that if $G$ is a $3$%
-dimensional nilpotent Lie group acting simply transitively on $\mathbb{R}%
^{3}$ by affine transformations, then $G$ contains a non-trivial central
translation. This answers affirmatively an old conjecture of Auslander in
dimension 3. However, an example of a simply transitive affine action of a
unipotent Lie group on $\mathbb{R}^{3}$ which contains no central
translation was constructed in \cite{fried}.

\begin{definition}
Let $\mathcal{G}$ be a Lie algebra with center $\mathcal{Z}$ $\left( 
\mathcal{G}\right) $, and define the lower central series of $\mathcal{G}$
to be the series of ideals $\left\{ \mathcal{C}^{i}\mathcal{G}\right\} $
given by $\mathcal{C}^{0}\mathcal{G=G},$ $\mathcal{C}^{1}\mathcal{G=}\left[ 
\mathcal{G},\mathcal{G}\right] ,$ and $\mathcal{C}^{i}\mathcal{G=}\left[ 
\mathcal{G},\mathcal{C}^{i-1}\mathcal{G}\right] $ for all $i\geq 1.$ If
there exists some integer $k\geq 1$ such that $\mathcal{C}^{k}\mathcal{G}%
\varsupsetneq \mathcal{\mathcal{C}}^{k+1}\mathcal{\mathcal{G}}=\left\{
0\right\} $, then we say that $\mathcal{G}$ is $k$-step nilpotent. In that
case, if in addition $\mathcal{Z}$ $\left( \mathcal{G}\right) =\mathcal{C}%
^{k}\mathcal{G},$ then we say that $\mathcal{G}$ is \emph{nonsingular}.
\end{definition}

\begin{proposition}
\label{prop3bis}Let $A$ be a left-symmetric algebra whose associated Lie
algebra $\mathcal{G}_{A}$ is $k$-step nilpotent. If $A$ satisfies (\ref{eq4}%
), then $\mathcal{C}^{k-1}\mathcal{G}_{A}\subseteq C\left( A\right) .$ In
particular $C\left( A\right) \neq \left\{ 0\right\} .$

If in addition $\mathcal{G}_{A}$ is nonsingular, then $C\left( A\right)
=Z\left( A\right) =\mathcal{Z}_{A}.$
\end{proposition}

\begin{proof}
By (\ref{eq4a}), we have $\left[ \mathcal{G}_{A},\mathcal{G}_{A}\right] =%
\left[ A,A\right] \subseteq T\left( A\right) ;$ and consequently%
\begin{equation*}
\left\{ 0\right\} \neq \mathcal{C}^{k-1}\mathcal{G}_{A}=\left[ \mathcal{G}%
_{A},\mathcal{G}_{A}\right] \cap \mathcal{Z}_{A}\subseteq T\left( A\right)
\cap \mathcal{Z}_{A}=C\left( A\right) .
\end{equation*}

Assume now that $\mathcal{G}_{A}$ is nonsingular. On the one hand, since $%
\mathcal{G}_{A}$ is nonsingular and $A$ satisfies (\ref{eq4}) we have $%
\mathcal{Z}_{A}=\mathcal{C}^{k-1}\mathcal{G}_{A}\subseteq \left[ \mathcal{G}%
_{A},\mathcal{G}_{A}\right] \subseteq T\left( A\right) ,$ from which we
deduce that $C\left( A\right) =T\left( A\right) \cap \mathcal{Z}_{A}=%
\mathcal{Z}_{A}.$ On the other hand, we have $C\left( A\right) \subseteq
Z\left( A\right) \subseteq \mathcal{Z}_{A}$.
\end{proof}

\bigskip

As an immediate but useful consequence of Proposition \ref{prop3bis}, we
have the following corollary which gives a method for classifying
left-symmetric products satisfying (\ref{eq4}) on low dimensional nilpotent
Lie algebras.

\begin{corollary}
Let $A$ be a left-symmetric algebra satisfying (\ref{eq4}) whose associated
Lie algebra is nilpotent. Then, $A$ can be obtained as a central extension
of a left-symmetric algebra satisfying (\ref{eq4}) of smaller dimension.
\end{corollary}

\subsection{Some comments on derivation algebras}

In this subsection we make some comments on other existing results
concerning derivation algebras. The first comment concerns a special family
of derivation algebras, the so-called inner derivation algebras. In \cite%
{medina}, a left-symmetric algebra $A$ over a field $\mathbb{F}$ is called
an \emph{inner derivation algebra} if all left (resp. right) multiplications
are inner derivations of the associated Lie algebra $\mathcal{G}_{A}.$ If $A$
is an inner derivation algebra, then it is easy to verify that the
left-symmetric product is given by 
\begin{equation*}
x\cdot y=ad_{f\left( x\right) }y=\left[ f\left( x\right) ,y\right] ,\ \ \ 
\text{for all }x,y\in A,
\end{equation*}%
where $f$ is an endomorphism of the vector space $A$ satisfying the
conditions:

\begin{enumerate}
\item $\left[ x,y\right] =\left[ f\left( x\right) ,y\right] +\left[
x,f\left( y\right) \right] ,$

\item $f\left( \left[ x,y\right] \right) -\left[ f\left( x\right) ,f\left(
y\right) \right] \in \mathcal{Z}_{A},\ \ $for all $x,y\in A.$
\end{enumerate}

\begin{proposition}
Let $A$ be an inner derivation algebra whose associated Lie algebra $%
\mathcal{G}_{A}$ is $2$-step nilpotent. Then, $A$ satisfies (\ref{eq4}).
\end{proposition}

\begin{proof}
Let $f$ be an endomorphism of the vector space $A$ such that $%
L_{x}=ad_{f\left( x\right) },$ for all $x\in A;$ and let $x,y,z\in A.$ From
the above second condition on $f,$ we have that $f\left( \left[ x,y\right]
\right) =\left[ f\left( x\right) ,f\left( y\right) \right] +z^{\prime },$
with $z^{\prime }\in \mathcal{Z}_{A}.$ It follows that%
\begin{eqnarray*}
\left[ x,y\right] \cdot z &=&\left[ f\left( \left[ x,y\right] \right) ,y%
\right] \\
&=&\left[ \left[ f\left( x\right) ,f\left( y\right) \right] +z^{\prime },z%
\right] \\
&=&0,
\end{eqnarray*}%
given that $\mathcal{G}_{A}$ is $2$-step nilpotent (i.e., $\left[ A,A\right]
\subseteq \mathcal{Z}_{A}$). Thus $A$ satisfies (\ref{eq4}), as desired.
\end{proof}

\bigskip

The second and third comments concern derivation algebras that are not
necessarily inner. These can be derived from the following result proved in 
\cite{medina}.

\begin{theorem}
\label{medina thm}A derivation algebra $A$ over a field $\mathbb{F}$ splits
uniquely as a direct sum of two-sided ideals $A_{0}$ and $A_{\ast }$ such
that $A_{0}$ is complete and contains the derived ideal $\left[ A,A\right] $
and $A_{\ast }$ is commutative with identity and contained in the center $%
Z\left( A\right) .$ Moreover, we have that $T\left( A\right) \subseteq A_{0}$
with $T\left( A\right) =\left\{ 0\right\} $ if and only if $A_{0}=\left\{
0\right\} .$
\end{theorem}

\begin{corollary}
\label{corolla1}Let $A$ be a noncommutative derivation algebra over a field $%
\mathbb{F}.$ Then $T\left( A\right) \neq \left\{ 0\right\} .$
\end{corollary}

\begin{proof}
By Theorem \ref{medina thm}, we have $\left[ A,A\right] \subseteq A_{0}.$
Since $A$ is noncommutative, we deduce that $A_{0}\neq \left\{ 0\right\} .$
Again, by Theorem \ref{medina thm}, this implies that $T\left( A\right) \neq
\left\{ 0\right\} .$
\end{proof}

\smallskip

Now, we consider a derivation algebra $A$ such that the center $\mathcal{Z}%
_{A}$ of its associated algebra $\mathcal{G}_{A}$ satisfies the inclusion $%
\mathcal{Z}_{A}\subseteq \left[ A,A\right] .$ In fact, there are a lot of
Lie algebras (even nilpotent of any step of nilpotency) which satisfy the
above inclusion. In that case, by Theorem \ref{medina thm} we have that $%
A_{\ast }\subseteq \mathcal{Z}\left( A\right) \subseteq \left[ A,A\right]
\subseteq A_{0},$ from which we deduce that $A_{\ast }=\left\{ 0\right\} ,$
that is, $A$ is complete. As a special case of this, let us consider a
derivation algebra $A$ with nonsingular $k$-step nilpotent associated Lie
algebra $\mathcal{G}_{A}$. In this case, we have%
\begin{equation*}
\mathcal{Z}_{A}=\mathcal{C}^{k-1}\mathcal{G}_{A}\mathcal{=Z}_{A}\cap \left[
A,A\right] \subseteq \left[ A,A\right] .
\end{equation*}

Thus, as an immediate consequence of Theorem \ref{medina thm}, we can also
state the following corollary.

\begin{corollary}
Let $A$ be a derivation algebra over a field $\mathbb{F}$ whose associated
Lie algebra $\mathcal{G}_{A}$ is nonsingular nilpotent. Then, $A$ is
complete.
\end{corollary}

\section{Radicals of a left-symmetric algebra}

The radical of an associative algebra $A$ is the unique nilpotent ideal of $%
A $ which is maximal, that is it contains all nilpotent ideals of $A.$ No
such ideal exists in an arbitrary nonassociative algebra, and so the radical
of such an algebra has never been defined. For the case of left-symmetric
algebras, different types of radicals have been defined. We shall consider
here three of them.

\subsection{The (Koszul) radical of a left-symmetric algebra}

The radical of a left-symmetric algebra was firstly defined by J. L. Koszul
(see \cite{helmstetter}). Given a left-symmetric algebra $A$ over a field $%
\mathbb{F},$ one defines \emph{the radical} $R\left( A\right) $ of $A$ to be
the largest left ideal contained in the subset%
\begin{equation*}
I\left( A\right) =\left\{ a\in A:tr\left( R_{a}\right) =0\right\} .
\end{equation*}

It turns out that $R\left( A\right) $ is nothing but the largest complete
left ideal of $A.$ This has been proved in \cite{helmstetter} for the case $%
\mathbb{F}=\mathbb{C},$ and in \cite{chang-kim-myung} for $\mathbb{F}=%
\mathbb{R}.$

In general, the radical of an arbitrary left-symmetric algebra is not a
two-sided ideal (cf. \cite{helmstetter}). However, as we will see later, we
can deduce from \cite{zelmanov} that the radical of a Novikov algebra over a
field of characteristic zero is a two-sided ideal. This is also the case for
derivation algebras (see \cite{medina}). In the case of a left-symmetric
algebra satisfying (\ref{eq4}), we have the following result.

\begin{theorem}
\label{thm1}Let $A$ be a left-symmetric algebra satisfying (\ref{eq4}).
Then, $R\left( A\right) $ is a two-sided ideal of $A$ containing the derived
Lie algebra $\left[ A,A\right] .$
\end{theorem}

\begin{proof}
By (\ref{eq4a}), the right multiplication $R_{x}$ is identically zero on $%
\left[ A,A\right] $ for all $x\in A.$ Therefore, $\left[ A,A\right] $ is a
complete subalgebra of $A.$ On the other hand, , by Proposition \ref%
{prop2bis}, $\left[ A,A\right] $ is a two-sided ideal of $A.$ It follows
that $\left[ A,A\right] \subseteq R\left( A\right) .$ This in turn implies
that $R\left( A\right) $ is a Lie ideal, and since it is a left ideal, we
deduce that $R\left( A\right) $ is a two-sided ideal containing $\left[ A,A%
\right] .$
\end{proof}

\begin{corollary}
Every left-symmetric algebra satisfying (\ref{eq4}) can be obtained as an
extension of a commutative associative algebra by a complete left-symmetric
algebra satisfying (\ref{eq4}).
\end{corollary}

\subsection{The left (resp. right)-nilpotent radical}

Let $A$ be a left-symmetric algebra over a field $\mathbb{F}$ of
characteristic zero, and $I$ a two-sided ideal of $A.$ We say that $I$ is 
\emph{left-nilpotent} if there exists some fixed integer $n\geq 1$ such that 
$L_{a_{1}}\cdots L_{a_{n}}=0$ for all $a_{i}\in I.$ It is not difficult to
show that every finite-dimensional left-symmetric algebra $A$ has a unique
maximal left-nilpotent ideal $L\left( A\right) ,$ called the \emph{left
radical} of $A$ (see \cite{chang-kim-myung}).

It is also clear that if $I$ is a left-nilpotent ideal of a left-symmetric
algebra $A,$ then the left multiplications $L_{a}$ are nilpotent for all $%
a\in I.$ On the other hand, it is well known that given a left-symmetric
algebra $A,$ then we have: all left multiplications $L_{a}$ are nilpotent if
and only if all the right multiplications $R_{a}$ are nilpotent and the
associated Lie algebra $\mathcal{G}_{A}$ is nilpotent (see \cite{scheuneman}%
; see also \cite{kim}, Theorem 2.1 and Theorem 2.2). We deduce from these
two facts that the left radical $L\left( A\right) $ of an arbitrary
left-symmetric algebra $A$ is a complete ideal; and consequently we have
that $L\left( A\right) \subseteq R\left( A\right) .$

\medskip

Similarly we define an ideal $I$ to be \emph{right-nilpotent} if there
exists some fixed integer $n\geq 1$ such that $R_{a_{1}}\cdots R_{a_{n}}=0$
for all $a_{i}\in I.$ It follows immediately that any right-nilpotent
algebra is complete. However, unlike the left-nilpotent case, the largest
right-nilpotent ideal need not exist for an arbitrary left-symmetric
algebra, because the sum of any two right-nilpotent ideals need not be
right-nilpotent. However, it was shown in \cite{zelmanov} that a Novikov
algebra $A$ has always a unique maximal right-nilpotent two-sided ideal $%
N\left( A\right) ,$ called the \emph{right radical} of $A.$ In the same
paper, it was also shown that $I\left( A\right) =\left\{ a\in A:tr\left(
R_{a}\right) =0\right\} $ is a two-sided ideal that is right-nilpotent. From
these two facts, we can deduce the following:

\begin{proposition}
\label{prop radicals}Let $A$ be a Novikov algebra over a field $\mathbb{F}$
of characteristic zero. Then, we have $N\left( A\right) =R\left( A\right)
=I\left( A\right) .$
\end{proposition}

\begin{proof}
As mentioned above, since $A$ is a Novikov algebra then the right radical $%
N\left( A\right) $ exists and the two-sided ideal $I\left( A\right) $ is
right-nilpotent. It follows from maximality of $N\left( A\right) $ that $%
I\left( A\right) \subseteq N\left( A\right) .$ On the other hand, since $%
N\left( A\right) $ is right-nilpotent, then right multiplications $R_{a}$
are nilpotent, and consequently $N\left( A\right) $ is complete. Thus, $%
N\left( A\right) \subseteq R\left( A\right) .$ But the definition of $%
R\left( A\right) $ says that $R\left( A\right) \subseteq I\left( A\right) .$
Hence, we have $I\left( A\right) \subseteq N\left( A\right) \subseteq
R\left( A\right) \subseteq I\left( A\right) ,$ as desired.
\end{proof}

\bigskip

As we mentioned above, a derivation algebra $A$ can always be uniquely
decomposed into a direct sum of two-sided ideals $A_{0}$ and $A_{\ast }$
such that $A_{0}$ is complete and contains the derived ideal $\left[ A,A%
\right] $ and $A_{\ast }$ is commutative with identity and contained in the
center $\mathcal{Z}\left( A\right) .$ In particular, for a derivation
algebra $A$ we have $A_{0}\subseteq R\left( A\right) .$ We give here an
example of a derivation algebra $A$ such that $A_{0}\varsubsetneq R\left(
A\right) .$

\begin{example}
Over the field $\mathbb{F}=\mathbb{R}$ or $\mathbb{C},$ consider the
two-dimensional left-symmetric algebra $A$ defined by following
multiplication table:$\ e_{1}\cdot e_{1}=e_{1},\ \ e_{1}\cdot
e_{2}=e_{2}\cdot e_{1}=e_{2}.$ Since the Lie algebra associated to $A$ is
commutative, it follows that $A$ is a derivation algebra (it is also Novikov
and satisfies (\ref{eq4})). It is now easy to see that $N\left( A\right)
=R\left( A\right) =\mathbb{F}e_{2}.$ When we look at $A$ as a derivation
algebra, we see that $A_{0}=\left\{ 0\right\} .$ This shows that $%
A_{0}\varsubsetneq R\left( A\right) .$ On the other hand, we should also
notice that, being a Novikov algebra, $A$ can be obtained as an extension of
the field $\mathbb{F}e_{1}$ by $N\left( A\right) =\mathbb{F}e_{2}.$
\end{example}

\section{\protect\bigskip Simple left-symmetric algebras}

An algebra $A$ over a filed $\mathbb{F}$ is called simple if it has no
proper two-sided ideal and $A$ is not the zero algebra of dimension $1.$
Therefore, since $A^{2}=A\cdot A$ is a two-sided ideal of $A,$ we have $%
A^{2}=A$ in case $A$ is simple.

In \cite{zelmanov}, the following result was proved.

\begin{theorem}
\label{zelmanov thm}A simple Novikov algebra $A$ over a field $\mathbb{F}$
of characteristic zero is isomorphic to $\mathbb{F}.$
\end{theorem}

It is worth mentioning that when the field $\mathbb{F}$ is not algebraically
closed, then simple Novikov algebras over $\mathbb{F}$ of dimension $\geq 2$
can exist. Here is an example of a two-dimensional simple Novikov algebra
over $\mathbb{R}.$

\begin{example}[A two-dimensional simple Novikov algebra over $\mathbb{R}$]
Over the field $\mathbb{F}=\mathbb{R}$ or $\mathbb{C},$ let us consider the
two-dimensional commutative associative algebra $A_{2,\mathbb{F}}$ defined
by the following multiplication table:$\ e_{1}\cdot e_{1}=e_{1},\ \
e_{1}\cdot e_{2}=e_{2}\cdot e_{1}=e_{2},\ \ e_{2}\cdot e_{2}=-e_{1}.$ Being
commutative, $A_{2,\mathbb{F}}$ is a Novikov algebra; and by setting $%
e_{1}^{\prime }=\frac{1}{2}\left( e_{1}+ie_{2}\right) ,\ \ e_{2}^{\prime }=%
\frac{1}{2}\left( e_{1}-ie_{2}\right) ,$ we can easily see that $A_{2,%
\mathbb{C}}$ is a direct sum of fields, that is $A_{2,\mathbb{C}}\cong 
\mathbb{C}\oplus \mathbb{C}.$ However, it is not difficult to show that $%
A_{2,\mathbb{R}}$ is simple.
\end{example}

\begin{remark}
It is worth pointing out that, in the example above, $A_{2,\mathbb{C}}$ is
nothing but the complexification of $A_{2,\mathbb{R}}.$ It follows that the
complexification of a simple left-symmetric algebra (even Novikov) need not
be simple.
\end{remark}

Anyway, the following theorem will show that $A_{2,\mathbb{R}}$ is the only
simple Novikov algebra over $\mathbb{R}$ of dimension $\geq 2.$ To prove the
theorem, we need to recall the notion of complexification of a
left-symmetric algebra.

Let $A$ be a real left-symmetric algebra of dimension $n,$ and let $A^{%
\mathbb{C}}$ denote the real vector space $A\oplus A.$ Let $J:A\oplus
A\rightarrow A\oplus A$ be the linear map on $A\oplus A$ defined by $J\left(
x,y\right) =\left( -y,x\right) .$

For $\alpha +i\beta \in \mathbb{C}$ and $x,x^{\prime },y,y^{\prime }\in A,$
we define%
\begin{equation}
\left( \alpha +i\beta \right) \left( x,y\right) =\left( \alpha x-\beta
y,\alpha y+\beta x\right)  \label{cmplx1}
\end{equation}%
and%
\begin{equation}
\left( x,y\right) \cdot \left( x^{\prime },y^{\prime }\right) =\left(
xx^{\prime }-yy^{\prime },xy^{\prime }+yx^{\prime }\right)  \label{cmplx2}
\end{equation}

We endow the set $A^{\mathbb{C}}$ with the componentwise addition,
multiplication by complex numbers defined by (\ref{cmplx1}), and the product
defined by (\ref{cmplx2}). It is then straightforward to verify that, if
endowed with the product defined by (\ref{cmplx2}), $A^{\mathbb{C}}$ becomes
a complex left-symmetric algebra of dimension $n$ that we call \emph{the
complexification}\textsl{\ }of $A.$ In that case, the left-symmetric algebra 
$A$ can be identified with the set of elements in $A^{\mathbb{C}}$ of the
form $\left( x,0\right) ,$ where $x\in A.$ Furthermore, if $e_{1},\ldots
,e_{n}$ is a basis in $A,$ then the elements $\left( e_{1},0\right) ,\ldots
,\left( e_{n},0\right) $ form a basis in the complex vector space $A^{%
\mathbb{C}}.$

\begin{theorem}
\label{thm2}Let $A$ be simple Novikov algebra over $\mathbb{R}.$ Then, $A$
is isomorphic to either $A_{2,\mathbb{R}}$ or the field $\mathbb{R}.$
\end{theorem}

\begin{proof}
Let $A^{\mathbb{C}}$ be the complexification of $A.$ Since $A$ is Novikov,
then it is easy to check that $A^{\mathbb{C}}$ is Novikov as well. Since, by
Proposition \ref{prop radicals} we have $N\left( A\right) =I\left( A\right)
, $ it follows that $N\left( A\right) ^{\mathbb{C}}=N\left( A^{\mathbb{C}%
}\right) ,$ where $N\left( A\right) ^{\mathbb{C}}$ is the complexification
of $N\left( A\right) $ and $N\left( A^{\mathbb{C}}\right) $ is the right
radical of $A^{\mathbb{C}}.$ But since $A$ is simple, we have either $%
N\left( A\right) =\left\{ 0\right\} $ or $N\left( A\right) =A.$ From this we
deduce that either $N\left( A\right) ^{\mathbb{C}}=\left\{ 0\right\} $ or $%
N\left( A\right) ^{\mathbb{C}}=A^{\mathbb{C}}.$ If $N\left( A\right) ^{%
\mathbb{C}}=A^{\mathbb{C}},$ then $A^{\mathbb{C}}$ is right-nilpotent. By
Proposition 1 of \cite{zelmanov}, $\left( A^{\mathbb{C}}\right) ^{2}$ is
nilpotent. If $\left( A^{\mathbb{C}}\right) ^{2}=\left\{ 0\right\} ,$ then $%
A^{2}=\left\{ 0\right\} $ which implies that $A$ is not simple, a
contradiction. Theus, we necessarily have $\left( A^{\mathbb{C}}\right)
^{2}\neq \left\{ 0\right\} .$ It follows that there exists $k\geq 1$ such
that $\left( A^{\mathbb{C}}\right) ^{2k}\varsupsetneq \left( A^{\mathbb{C}%
}\right) ^{2k+2}=\left\{ 0\right\} .$ This implies that $\left( A^{\mathbb{C}%
}\right) ^{2k}$ is a non-trivial two-sided ideal of $A^{\mathbb{C}}.$ It
follows that $A^{2k}$ is a non-trivial two-sided ideal of $A,$ which leads
to a contradiction since $A$ is assumed to be simple. Thus, $N\left(
A\right) ^{\mathbb{C}}=\left\{ 0\right\} .$ In this case, Proposition 2 of 
\cite{zelmanov} tells us that $A^{\mathbb{C}}$ is a direct sum of fields. On
the other hand, by Lemma 2.10 of \cite{bai}, $A^{\mathbb{C}}$ is simple or a
direct sum of two simple ideals. It follows that $A^{\mathbb{C}}$ is
isomorphic to either the field $\mathbb{C}$ or the direct sum $\mathbb{C}%
\oplus \mathbb{C}.$ Now, by applying Corollary 3.3 of \cite{bai}, we deduce
that $A$ is isomorphic to either the field $\mathbb{R}$ or $A_{2,\mathbb{R}%
}, $ as desired.
\end{proof}

\bigskip

For derivation algebras and left-symmetric algebras satisfying (\ref{eq4}),
we have the following immediate consequence of Proposition \ref{prop2bis}
and Corollary \ref{corolla1}.

\begin{lemma}
\label{lemma3}A simple left-symmetric algebra over a field $\mathbb{F}$
which is derivation or satisfying (\ref{eq4}) is necessarily commutative.
\end{lemma}

If the field is not algebraically closed, we have the following immediate
consequence of Theorem \ref{thm2} and Lemma \ref{lemma3}.

\begin{proposition}
\label{prop4}A simple left-symmetric algebra over $\mathbb{R}$ which is
derivation or satisfying (\ref{eq4}) is isomorphic to either $A_{2,\mathbb{R}%
}$ or the field $\mathbb{R}.$
\end{proposition}

If the field is algebraically closed, we have the following immediate
consequence of Theorem \ref{zelmanov thm} and Lemma \ref{lemma3}.

\begin{proposition}
\label{prop5}Let $A$ be a simple left-symmetric algebra over an
algebraically closed field $\mathbb{F}$ which is derivation or satisfying (%
\ref{eq4}). Then, $A$ is isomorphic to $\mathbb{F}.$
\end{proposition}

We close this section with the following propositions concerning
completeness of simple left-symmetric algebras.

\begin{proposition}
A complete Novikov algebra over a field of characteristic zero is not simple.
\end{proposition}

\begin{proof}
Let $A$ be a complete Novikov algebra over a field $\mathbb{F}$ of
characteristic zero. Since $A$ is complete, then $A$ is right-nilpotent. By
Proposition 1 of \cite{zelmanov}, $A^{2}$ is nilpotent. If $A^{2}=\left\{
0\right\} ,$ then $A$ is obviously not simple. If $A^{2}\neq \left\{
0\right\} ,$ then there exists $k\geq 1$ such that $A^{2k}\varsupsetneq
A^{2k+2}=\left\{ 0\right\} .$ Thus $A$ contains the non-trivial two-sided
ideal $A^{2k},$ and hence $A$ is not simple, as desired.
\end{proof}

\begin{proposition}
Let $A$ be a complete left-symmetric algebra over the field $\mathbb{F}=%
\mathbb{R}$ or $\mathbb{C}.$ If $A$ is derivation or satisfies (\ref{eq4}),
then $A$ is not simple.
\end{proposition}

\begin{proof}
Suppose to the contrary that $A$ is simple. In the case $\mathbb{F}=\mathbb{R%
},$ we derive from Proposition \ref{prop4} that $A$ is isomorphic to either $%
A_{2,\mathbb{R}}$ or the field $\mathbb{R}.$ In the case $\mathbb{F}=\mathbb{%
C},$ we derive from Proposition \ref{prop5} that $A$ is isomorphic to the
field $\mathbb{F}=\mathbb{C}.$ In both cases, this leads to a contradiction
since $A_{2,\mathbb{R}}$ and the field $\mathbb{F}$ are not complete. It
follows that $A$ is not simple, as desired.
\end{proof}

\bigskip

Department of Mathematics, College of Science,

King Saud University, P.O.Box 2455, Riyadh 11451

Saudi Arabia

\smallskip

E-mail: mguediri@ksu.edu.sa

\end{document}